\documentclass[11pt,reqno]{amsart}
\usepackage[utf8]{inputenc}
\usepackage{multicol}
\usepackage{caption}
\usepackage{subcaption}
\usepackage{tabularx}
\usepackage{capt-of}
\usepackage{tikz}
\usepackage{float}
\usepackage{lineno}
\usepackage{graphicx}
\usepackage{float}
\usepackage{comment}
\usepackage[shortlabels]{enumitem}
\usepackage{amsthm,amsmath, amsfonts}
\usepackage{enumitem}
\usepackage{fancyhdr}
\usepackage{appendix}
\usepackage{placeins}
\usetikzlibrary{calc}
\usepackage{lineno}
\usepackage{bm}
\usepackage{csquotes}

\usepackage[backend=biber,backref=true]{biblatex}
\usepackage{hyperref}
\usepackage{xcolor}
\hypersetup{
    colorlinks,
    linkcolor={red!50!black},
    citecolor={blue!50!black},
    urlcolor={blue!80!black}
}
\usepackage{url}

\numberwithin{equation}{section}

\newtheorem{theorem}[equation]{Theorem}

\newtheorem{proposition}[equation]{Proposition}

\newtheorem{hyp}[equation]{Hypothesis}
\newtheorem{ques}[equation]{Question}

\theoremstyle{definition}
\newtheorem{definition}[equation]{Definition}
\newtheorem{remark}[equation]{Remark}

\newtheorem{exam}[equation]{Example}

\DeclareMathOperator{\ch}{ch}
\DeclareMathOperator{\down}{Down}

\DeclareMathOperator{\CH}{CH}

\PassOptionsToPackage{hyphens}{url}

\def\mc{\mathcal}
\def\R{\mathbb{R}}

\newcommand{\mbf}{\mathbf}

\title{Representation of convex geometries of convex dimension 3 by spheres}

\author{Kira Adaricheva}

\author{Arav Agarwal}

\author{Na'ama Nevo}



\date{\today}

\begin{document}

\setlength{\footskip}{13.0pt}.

\thanks{The work on this paper was initiated while the second and third authors attended the New York Discrete Mathematics REU in the summer of 2022 and were mentored by the first author. The REU was supported by NSF grant \# 2051026 and led by PI Adam Sheffer and Co-PI Pablo Soberon Bravo (both CUNY Baruch College).
We appreciate the welcoming atmosphere of the Mathematics Department at Baruch College that generously hosted the REU, as well as the support of other mentors and student participants.}

\begin{abstract} A convex geometry is a closure system satisfying the anti-exchange property. This paper, following the work of Adaricheva and Bolat (2019) and the Polymath REU (2020), continues to investigate representations of convex geometries with small convex dimension by convex shapes on the plane and in spaces of higher dimension. In particular, we answer in the negative the question raised by Polymath REU (2020): whether every convex geometry of $cdim=3$ is representable by the circles on the plane. We show there are geometries of $cdim=3$ that cannot be represented by spheres in any $\mathbb{R}^k$, and this connects to posets not representable by spheres from the paper of Felsner, Fishburn and Trotter (1999). On the positive side, we use the result of Kincses (2015) to show that every finite poset is an ellipsoid order.
\end{abstract}

\maketitle

\section{Introduction}

This paper addresses a question raised in the first paper of the convex geometries team \cite{Poly20} developed during the Polymath REU-2020 project, see more details about the project in \cite{Lemons21}.

\begin{ques}\label{cdim3}
    Is it possible to represent all finite convex geometries of convex dimension 3 by circles in the plane using the convex hull operator acting on circles?
\end{ques}

Convex geometries are closure systems whose closure operator satisfies the anti-exchange property. \emph{Convex dimension} ($cdim$) refers to a certain parameter of such closure systems defined by Edelman and Jamison in \cite{EdJa85}.

The representation of convex geometries by the circles on the plane was first introduced by Cz\'edli \cite{Cz14}, and the first convex geometry which cannot be represented that way was found in Adaricheva and Bolat \cite{AdBo19}. The obstruction to such representation came in the form of the Weak Carousel Property that failed for some convex geometry on a 5-element set and whose convex dimension = 6.

In \cite{Poly20} it was discovered that among 672 non-isomorphic geometries on a 5-element set there are seven that cannot be represented due to the Triangle Property.
In particular, some among these seven had $cdim=4,5$. At the same time, all geometries of $cdim=3$ on a 5-element set had circular representations on the plane.

This left open Question \ref{cdim3} regarding the representation of arbitrary convex geometries of $cdim =3$: either by circles on the plane, or by the spheres in spaces $\mathbb{R}^k$ of dimension $k>2$. In this paper, we answer the question in the negative by reducing this problem to another problem that  brought considerable attention of many combinatorialists and order specialists in the 1990s.

In 1989, Brightwell and Winkler \cite{BW89} asked whether
every finite poset is a sphere order and suggested that the answer was negative. As usual, by a \emph{poset} we assume a pair $(X,\leq)$, where $X$ is a set and $\leq$ is a partial order: a reflexive, anti-symmetric and transitive binary relation on $X$. A poset is a \emph{sphere order}, when every $x\in X$ can be represented by a sphere $F(x)$ in $\mathbb{R}^k$ so that $x\leq y$ iff $F(x)\subseteq F(y)$.

Indeed, the question was solved in the negative by Felsner, Fishburn and Trotter \cite{FFT} by presenting a 3-dimensional poset that cannot be a sphere order in any space $\mathbb{R}^k$. Note that \emph{poset dimension} is defined as the smallest number $t$ such that the poset embeds into a product (with component-wise ordering) of $t$ chains. Equivalently, the partial order is recovered as an intersection of at most $t$ of its linear extensions. It is well-known that 2-dimensional posets are representable by closed intervals in $\mathbb{R}$, i.e., by spheres in one-dimensional space.

We observe that the closure space of any convex geometry can be thought of as a poset of closed sets of its associated closure operator. However, poset dimension and convex dimension of a convex geometry are fairly different parameters. Recently, the relation was studied in Knauer and Trotter \cite{KT23}, who presented a series of convex geometries which have poset dimension = 3, while their convex dimension grows unboundedly.

In the Propositions of Section~\ref{sec:cg_existence}, we show that the poset from \cite{FFT} which is not a sphere order can be made isomorphic to the poset of join-irreducible elements of some convex geometry of $cdim=3$. In Proposition \ref{bridge} we connect the elements of the base set of the convex geometry with join-irreducible elements of its closure space, which allows to conclude that the convex geometry will not be represented by the spheres in any $\mathbb{R}^k$, proving our main result in Theorem~\ref{main}.

Note that the result of \cite{FFT} does not identify the smallest size of a poset which is not a sphere order. This suggests a question about \emph{the smallest size $\rho$ of (the base set of) a convex geometry of $cdim=3$ that is not representable by spheres in any} $\mathbb{R}^n$.

We propose that at least $\rho > 6$:

\begin{hyp}
    Every convex geometry of $cdim=3$ on a 6-element set is representable by circles on the plane.
\end{hyp}

Note that according to the Online Encyclopedia of Integer Sequences (OEIS.org), there are almost 200,000 non-isomorphic convex geometries (equivalently, \emph{antimatroids}) on a 6-element set, with the exact number being given in sequence A224913.

Much less was known about \emph{ellipsoid orders}, i.e., finite posets that are representable by ellipsoids in some $\mathbb{R}^k$. We use a result of Kincses \cite{Kin17} regarding the representation of convex geometries by ellipsoids using the convex hull operator for ellipsoids to conclude that every finite poset is an ellipsoid order.

\section{Terminology and Known Results}

A convex geometry is a special case of a closure system. It can be defined through a closure operator, or by means of a closure space.

\begin{definition}\label{def:closure}
Let $X$ be a set. A mapping $\varphi \colon 2^X \to 2^X$ is called a \emph{closure operator}, if for all $Y, Z \in 2^X$:
\begin{enumerate}[label=(\arabic*), noitemsep]
    \item $Y \subseteq \varphi(Y)$,
    \item if $Y \subseteq Z$ then $\varphi(Y) \subseteq \varphi(Z)$,
    \item $\varphi(\varphi(Y)) = \varphi(Y)$.
\end{enumerate}
 
A subset $Y \subseteq X$ is \emph{closed} if $\varphi(Y) = Y$. The pair $(X,\varphi)$, where $\varphi$ is a closure operator, is called a \emph{closure system}.
\end{definition}

\begin{definition}\label{def:alignment}
Given any (finite) set $X$, a \textit{closure space} on $X$ is a family
$\mathcal{F}$ of subsets of $X$ which satisfies two properties:
\begin{enumerate}[label=(\arabic*), noitemsep]
    \item $X \in \mathcal{F}$,
    \item if $Y, Z \in \mathcal{F}$ then $Y \cap Z \in \mathcal{F}$.
\end{enumerate}
\end{definition}
\noindent{Closure systems are dual to closure spaces in the following sense.}

\noindent If $(X, \varphi)$ is a closure system, one can define a family of closed sets $\mathcal{F}_\varphi :=\{Y \subseteq X : \varphi(Y) = Y\}$. Then $\mathcal{F}_\varphi$ is a closure space.\\
\noindent If $\mathcal{F}$ is a closure space, then define $\varphi_{\mc F}: 2^X \rightarrow 2^X$ in the
following manner: for all $Y \subseteq X$, let $\varphi_{\mc F}(Y):=\bigcap  \{Z \in \mathcal{F}: Y \subseteq Z\}$. Then $(X, \varphi)$ is a closure system.

\begin{definition}\label{def:cg}
A closure system $(X,\varphi)$ is called a \emph{convex geometry} if
\begin{enumerate}[label=(\arabic*), noitemsep]
    \item $\varphi(\emptyset) = \emptyset$,
    \item for any closed set $Y\subseteq X$ and any distinct points $x,y \in X\setminus Y,$ if $x \in
    \varphi(Y \cup \{y\})$ then $y \not\in \varphi(Y \cup \{x\})$.
\end{enumerate}
\end{definition}
Property (2) above is called the \emph{Anti-Exchange Property}.

\noindent We can use duality between closure operators and closure spaces to provide another definition of a convex geometry.
\begin{definition}
\label{cg_alignment}
A closure system $(X,\varphi)$ is a convex geometry iff the corresponding closure space
$\mathcal{F}_\varphi$ satisfies the following two properties:
\begin{enumerate}[label=(\arabic*), noitemsep]
    \item $\emptyset \in \mathcal{F}_\varphi$,
    \item if $Y \in \mathcal{F}_\varphi$ and $Y \neq X$, then there exists $a\in X\setminus Y$
    such that $Y\cup\{a\} \in \mathcal{F}_\varphi$.
\end{enumerate}
\end{definition}

We now need to discuss an important parameter of convex geometries known as \emph{convex dimension}, first introduced in \cite{EdJa85}.

\begin{definition}\label{def:mon}
A closure space $\mathcal{F}$ is called \textit{monotone} if sets of $\mathcal{F}$ form a chain under inclusion.\\
\noindent (See Definition~\ref{def:chain} for a chain or linear order.)
\end{definition}

\begin{remark}\label{rem:linearCG}
Note that, due to Definition \ref{cg_alignment}, a monotone closure space $\mathcal{F}$ is a convex geometry iff $\mathcal{F}$ has exactly $|X|+1$ subsets of $X$: $\emptyset \subset \{x_1\} \subset \{x_1,x_2\} \subset \{x_1,x_2,x_3\}\subset \dots \{x_1, x_2,\dots, x_{n-1}\}\subset X=\{x_1,\dots x_{n-1},x_n\}$. Therefore, every monotone convex geometry, or \emph{linear geometry}, on $X$ is uniquely associated with some linear order on $X$: $x_1<x_2<\dots <x_{n-1} < x_n$.
\end{remark}

\begin{definition}\label{join}
Given two closure spaces $\mathcal{F}, \mathcal{K}$ on the same base set $X$, their \textit{join} $\mathcal{F}\vee \mathcal{K}$ is defined to be
the smallest closure space $\mathcal{T}$ such that $\mathcal{F}, \mathcal{K} \subseteq \mathcal{T}$. 
\end{definition}

It can be easily verified that 
$\mathcal{F}\vee \mathcal{K}=\{ F\cap K : F \in
\mathcal{F} \text{ and }  K \in \mathcal{K} \}$.

\noindent A known result in \cite{EdJa85} about joins of convex geometries follows.
\begin{theorem}\label{CGjoin}
Let $X$ be a finite set.
\begin{itemize}
    \item[(1)] If closure spaces $\mathcal{F}, \mathcal{K}$ on $X$ are convex geometries, then $\mathcal{F}\vee \mathcal{K}$ is a convex geometry as well.
    \item[(2)] Closure space $\mathcal{F}$ of any convex geometry on set $X$  can be expressed as the join of some collection of monotone
convex geometries on the same base set. 
\end{itemize}

\end{theorem}
This motivates the following definition.
\begin{definition}\cite{EdJa85}\label{cdim}
The \emph{convex dimension} of a convex geometry $\mbf G=(X,\varphi)$ is the minimal number $k$ such that closure space $\mathcal{F}_\varphi$ can be expressed as the join of $k$ monotone convex geometries on set $X$.
\end{definition}

To compute the convex dimension of a convex geometry, we can examine maximal cardinality antichains of meet-irreducibles in its closure space $\mathcal{F}_\varphi$, as discussed in Edelman and Saks \cite{EdSa88}. Thus, informally, the $cdim$ parameter of a convex geometry represents the diversity of closed sets with respect to the closure operator $\varphi$.

\

\noindent A particular example of a closure operator on a set is the \emph{convex hull operator}, where the base set $X$ is a set of points in Euclidean space $\mathbb{R}^k$.

\begin{definition}
\ 
\begin{itemize}
    \item[(1)] A set $S$ in $\mathbb{R}^k$ is called  \emph{convex} if for any two points $p,q \in S$, the line segment connecting $p$ and $q$ is also contained in $S$.
    \item[(2)] Given a set $S$ of points in $\mathbb R^k$, the \emph{convex hull} of $S$, denoted $\CH(S)$, is the intersection of all convex sets in $\mathbb R ^k$ which contain $S$. That is, it is the smallest convex set containing $S$.
\end{itemize}
\end{definition}

\noindent Comparing with Definition~\ref{def:closure}, we see that $\CH$ is a closure operator acting on $\R^k$.

Finally, we recall the definition of the convex hull operator for spheres introduced in \cite{Cz14}. If $x$ is a sphere in $\mathbb{R}^k$, then by $\tilde{x}$ we denote the set of points belonging to $x$. It is allowed that a sphere has a radius $0$, in which case it is a point.

\begin{definition}
Let $X$ be a finite set of spheres in $\mathbb{R}^k$. Define the convex hull operator for spheres, $\ch_s : 2^X \rightarrow 2^X$, as follows:
\[
\ch_s(Y) =\{x \in X: \tilde{x}\subseteq \CH\bigg(\bigcup_{y \in Y}\tilde{y}\bigg)\},
\]
for any $Y\in 2^X$.
\end{definition}

\noindent See the figure below for an illustration of the $\ch_s$ operator in $\R^2$. Observe that $\ch_s(\{a,b,c\}) = \{a,b,c,d,e\}$ and $\ch_s(\{a,c\}) = \{a,c,e\}$.

\begin{figure}[H]
\centering
\begin{subfigure}{.5\textwidth}
  \centering
  \includegraphics[width=.8\linewidth]{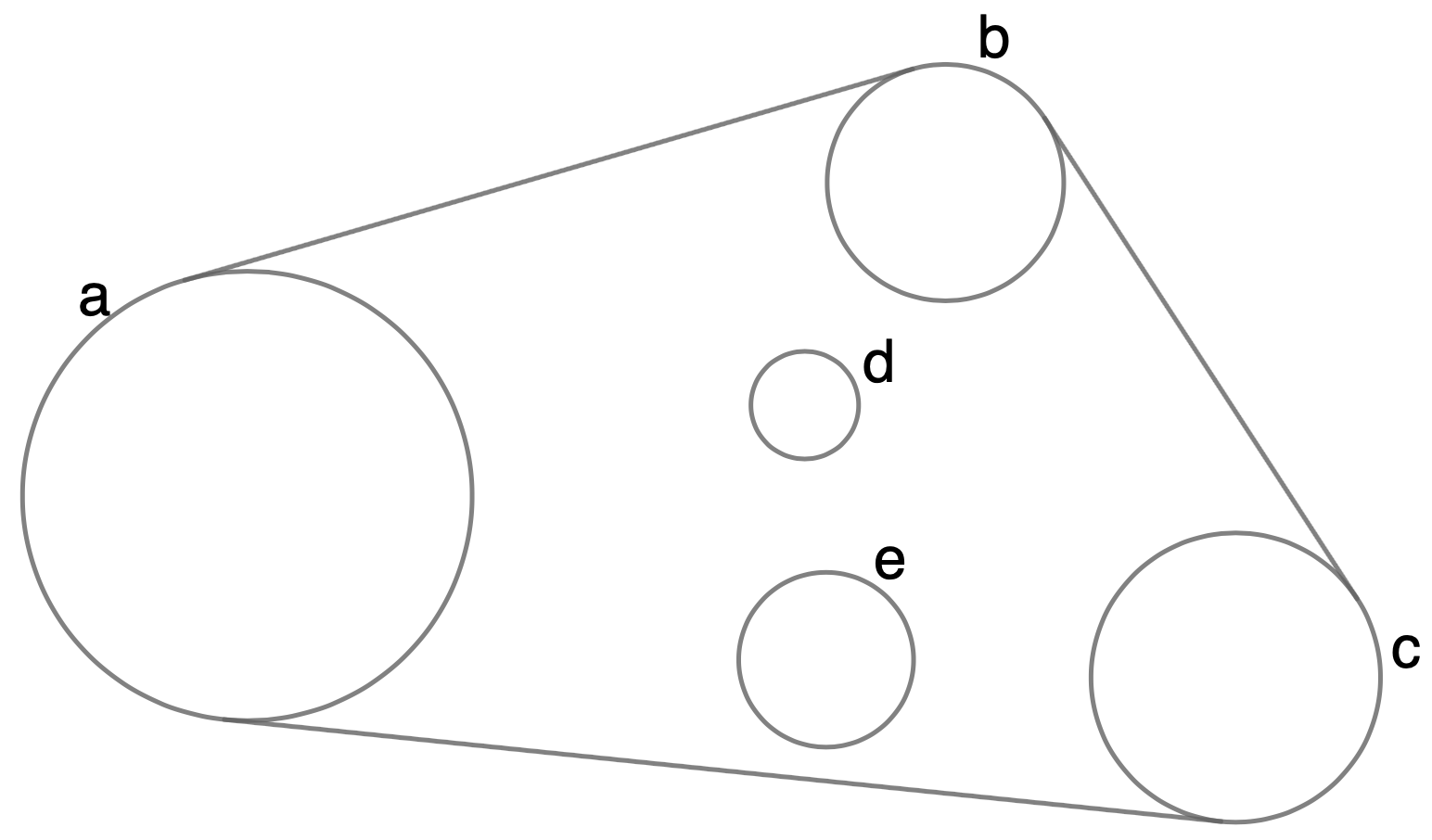}
  \caption{Circles $d, e$ are in the convex hull of $a, b, c$}
  \label{fig:sub1}
\end{subfigure}%
\begin{subfigure}{.5\textwidth}
  \centering
  \includegraphics[width=.8\linewidth]{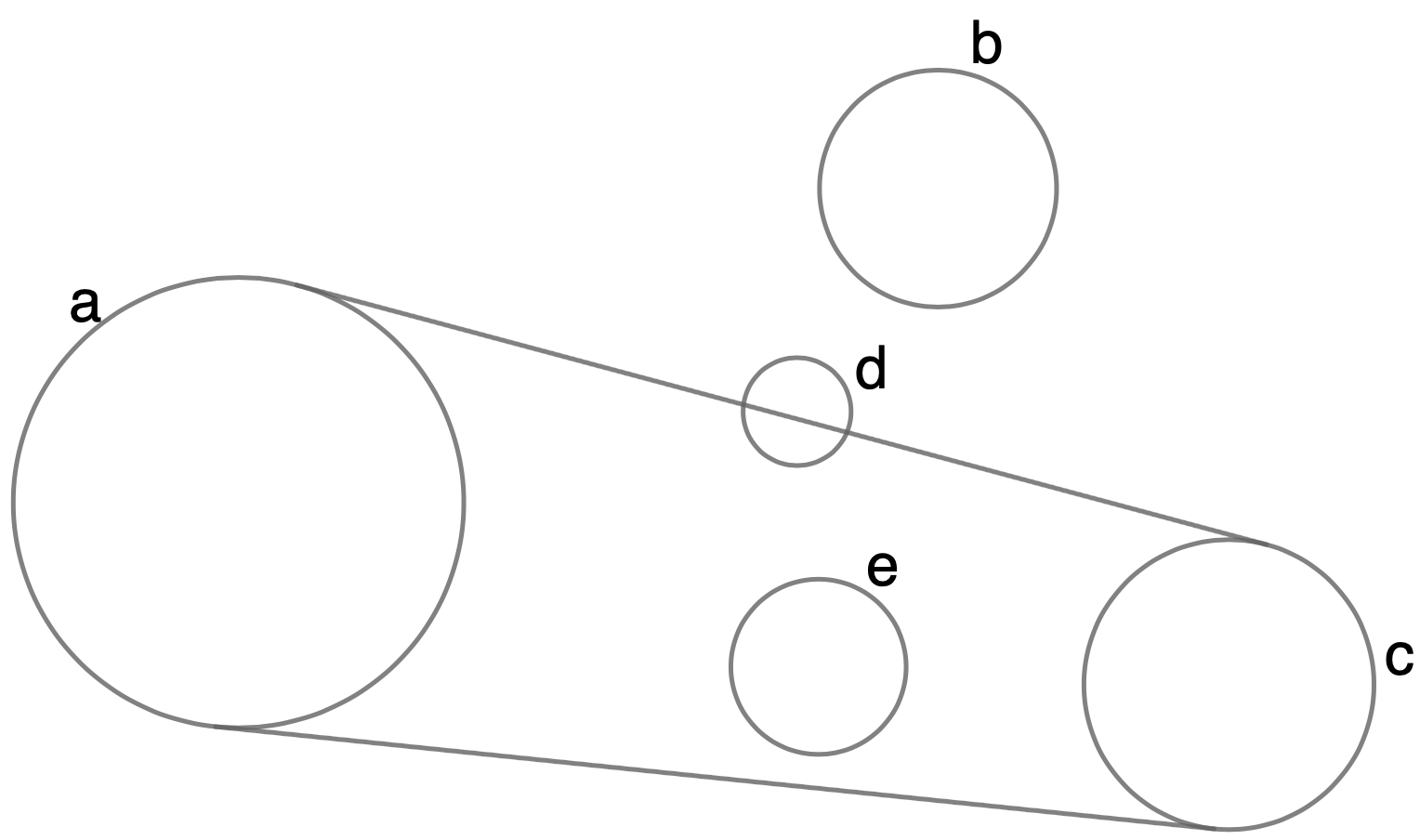}
  \caption{Circle $e$ is in convex hull of $a, c$ but $b, d$ are not}
  \label{fig:sub2}
\end{subfigure}
\caption{Convex hull operator for spheres}
\label{fig:test}
\end{figure}

It was established in \cite{Cz14} that the closure operator $\ch_s$ satisfies the Anti-exchange Property. Therefore, the closure system $(X,\ch_s)$ on the set of spheres $X$ in $\mathbb{R}^k$ is a convex geometry.

We say that a finite convex geometry $(X,\varphi)$ is \emph{represented by spheres in $\mathbb{R}^k$}, when there are $|X|$ spheres in $\mathbb{R}^k$ such that the action of $\varphi$ is identical to $\ch_s$. We can similarly define convex hull operators in $\mathbb{R}^k$ on convex shapes different from spheres, such as ellipsoids.

The survey paper \cite{EdJa85} was instrumental to start off the study of finite convex geometries, which are also the dual systems to \emph{antimatroids}. The study of infinite convex geometries was initiated in Adaricheva, Gorbunov and Tumanov \cite{AGT03}. To see the development of the topic, including infinite convex geometries, one needs to consult the more recent survey by Adaricheva and Nation \cite{AdNa16}. 

In \cite{Cz14} the representation of finite convex geometries was proposed by interpreting elements of base set $X$ as circles in the plane, and closure operator $\varphi$ as a convex hull operator acting on circles. The main result of the paper is that all finite convex geometries with \emph{convex dimension} at most 2 can be represented by circles on the plane.

In \cite{AdBo19} it was found that there is an obstruction for representation of convex geometries by circles on the plane, which allowed the authors to build an example of a convex geometry on a 5-element set of $cdim=6$. More obstructions for representation of geometries by circles were found in two papers written by subgroups of PolyMath-2020 team \cite{Poly20, Poly20-2}.

\

Passing to the terminology on posets, we recall that a set $X$ with a binary relation $\leq $ is called a partially ordered set (poset) if $\leq$ is reflexive, anti-symmetric and transitive. An important example of a poset is the family $\mathcal{F}$ of closed sets of a closure operator with relation $\subseteq$.

\begin{definition}\label{def:chain}
    A \emph{linear order}, or \emph{chain}, is a partial order $(X, \leq)$ where any two elements of $X$ are comparable, that is either $u \leq v$ or $v \leq u$, for any $u,v \in X$. We freely use the terms linear order and chain interchangeably.
\end{definition}

A relevant parameter of a poset is the \emph{the order dimension}, also known as the Dushnik-Miller dimension. Recall that a linear extension $(X,\leq^*)$ of poset $(X,\leq)$ is a poset with $\leq \subseteq \leq^*$, where $\leq^*$ is a linear order.

\begin{definition}
    The order dimension of a poset $(X,\leq)$ is the least integer $t$ for which we have a family of $t$ linear extensions $\leq_1, \ldots, \leq_t$ of $\leq$ such that $\leq = \bigcap_{i=1}^t \leq_i$.
\end{definition}

Equivalently, order dimension is the minimal number of chains such that the poset embeds into their direct product. For a comprehensive monograph on the topic, see Trotter \cite{Tr92}.

\begin{definition}
    Given a poset $\mbf P = (X,\leq)$, a function $F$ which assigns to each $x \in X$ a set $F(x)$ is called an inclusion representation of $\mbf P$ when $x \leq y$ if and only if $F(x) \subseteq F(y)$.
\end{definition}

In other words, an inclusion representation is the mapping which realizes $\mbf P$ as an inclusion order of some objects.

\begin{exam}\label{exam:down_set}
    Any poset has an inclusion representation. Given $\mbf P = (X, \leq)$, we can associate to each element $x \in X$ the set \[
    F(x)=\downarrow x = \{y \in X \, | \, y \leq x\}.
    \]
    This is the \emph{down-set} generated by the element $x$. A down-set in general is a set $S \subseteq X$ such that if $s \in S$ and $x \leq s$, then $x \in S$.
    
    The transitivity of $\leq$ tells us that 
    \[
    \downarrow a \subseteq \downarrow b \text{ iff } a \leq b.\]
\end{exam}

Indeed, the idea of the above example is fundamental, and key to Birkhoff's representation theorem, often referred to as the fundamental theorem of finite distributive lattices. 

\begin{theorem}\cite{birkhoff}\label{thm:birkhoff}
    The lattice of down-sets of a poset is distributive. Any finite distributive lattice $L$ is isomorphic to the lattice of down sets of the partial order of the join-irreducible elements of $L$.
\end{theorem}

In the theorem, the join-irreducible elements of the lattice of all down-sets of the poset are precisely the down-sets generated by singletons as illustrated in Example~\ref{exam:down_set}. Note that mapping $\varphi: Y\mapsto \downarrow Y$, that maps any subset $Y$ of partially ordered set $(X,\leq)$ into smallest down-set $\downarrow Y$ containing $Y$ is a closure operator on $X$ satisfying the Anti-Exchange Property. Thus, the lattice in Theorem \ref{thm:birkhoff} is also the lattice of closed sets of this closure operator, and finite distributive lattices are convex geometries.

\

\noindent Similar to the representation of convex geometries by convex shapes, we have the concept of representation of posets by spheres.

\begin{definition}
    A poset $\mbf P= (X, \leq)$ is a sphere order if there exists $k \geq 1$ such that $\mbf P$ has an inclusion representation using spheres in $\R^k$. That is, $F(x)$ for any $x \in X$ is required to be a sphere in $\R^k$.
\end{definition}

Before discussing results concerning representation of posets as sphere orders, we establish the key connection between representation of convex geometries by spheres and posets as sphere orders.

\begin{proposition}\label{bridge}
    Suppose a convex geometry $(X,\varphi)$ is represented by spheres in some $\R ^k$. Then, the poset of join-irreducible elements of the associated closure space $\mathcal{F}_\varphi$ is a sphere order in $\R ^k$.
\end{proposition}
\begin{proof}
By the definition of representation of convex geometries by spheres, each element $x\in X$ is given by some sphere $F(x)$ in $\R^k$ so that $\varphi$ acts on spheres as the $\ch_s$ operator. It is well-known that in \emph{standard} closure systems $(X, \varphi)$ there is a one-to-one correspondence between elements of $X$ and $\mathrm{Ji}(\mc F_\varphi)$ (see, for example, \cite[Lemma 4-2.8]{AdNa16II}), where $\mathrm{Ji}(\mc F_\varphi)$ denotes the set of join-irreducible elements of the closure space $\mathcal{F}_\varphi$ associated to $X$. Convex geometries are standard closure systems, so this correspondence works in convex geometries. Specifically, the one-to-one correspondence between $X$ and $\mathrm{Ji}(\mc F_ \varphi)$ is given by $x\mapsto \varphi(\{x\})$.
Ordering by set inclusion, we obtain the poset $(\mathrm{Ji}(\mc F_\varphi), \leq) = (\{\varphi(\{u\}):u\in X\}, \leq)$, and we now show it is a sphere order.

Consider $\varphi(\{x\}) \in \mathcal{F}_\varphi$. Since we represented the geometry by spheres, in terms of the sphere representation $\varphi(\{x\})$ is exactly $\ch_s(F(x))$. By definition of $\ch_s$, we obtain 
\[\ch_s(F(x)) = \{F(y) \, | \, F(y) \subseteq F(x), y \in X\}.\]
Indeed, we have arrived at the case of Example~\ref{exam:down_set}, for if we consider the inclusion order of the spheres, we observe that 
$$\ch_s(F(x)) = \, \downarrow F(x),$$
and we have $\downarrow F(u) \subseteq \downarrow F(v)$ iff $F(u) \subseteq F(v)$ for any $u,v \in X$. We can now write for any $u, v \in X$:
$$ \varphi(\{u\})\leq \varphi(\{v\}) \iff \ch_s(F(u)) \subseteq \ch_s(F(v)) \iff \downarrow F(u) \subseteq \downarrow F(v) \iff F(u) \subseteq F(v).$$

Hence, we have shown that $\varphi(\{u\})\leq \varphi(\{v\})$ iff $F(u)\subseteq F(v)$. Therefore, the sphere representation of a convex geometry $(X,\varphi)$ simultaneously provides us the representation of the poset of join-irreducibles $(\{\varphi(\{u\}):u\in X\}, \leq)$ of the closure space $\mathcal{F}_\varphi$, proving the poset is a sphere order.
\end{proof}

As per the record in \cite{FFT}, the question of whether every finite 3-dimensional poset has an inclusion representation using circles in $\R^2$ was raised by Fishburn and Trotter at the Banff conference of 1984. In Sidney et al.\! \cite{Sy88} a finite 4-dimensional poset was found that was not a \emph{circle order} (i.e., not represented by spheres on the plane), and in Scheinerman and Wierman \cite{Sc88} it was shown by a Ramsey theoretic argument that the countably infinite 3-dimensional poset $\mathbb{Z}^3$ is not a circle order.

These results prompted the following more general question:

\begin{ques}\cite{BW89}
    Is every finite 3-dimensional poset representable as an inclusion order of spheres in some $k$-dimensional space?
\end{ques}

In \cite{FFT} we find the culmination of the search to answer this question, in the negative.

With the next simple definition in hand, we can formally state the main theorem from \cite{FFT}.

\begin{definition}\label{def:n_poset}
    For positive integers $n, t$, let $\mbf n$ denote the chain \[
    0 < 1 < \cdots < n-1
    \]
    and $\mbf n^t$ the cartesian chain product of $t$ copies of $\mbf n$, so that we obtain the following canonical partial ordering on $\mbf n^t$:
    \[(a_1, a_2, \ldots, a_t) \leq (b_1, b_2, \ldots, b_t) \text{ iff } a_i \leq b_i \,\, \text{for all } i.\]
    
\end{definition}

\begin{theorem}[2.1 of \cite{FFT}]\label{thm:FFT_poset}
 There exists an integer $n_0$ such that if $n\geq n_0$, the finite 3-dimensional poset $\mbf n^3$ is not a sphere order.  
\end{theorem}

In Section \ref{sec:cg_existence}, we will use this theorem and Proposition \ref{bridge} to show that not every convex geometry of $cdim=3$ is representable by spheres.

We note that a convex geometry $(X,\varphi)$ can be thought of as a poset $(\mathcal{F}_\varphi,\subseteq)$, which could be measured using order dimension. The relationship between order dimension and the convex dimension pertinent to convex geometries was studied in \cite{KT23}. While in 2-dimensional geometries the convex dimension is the same and equals 2, the picture is quite different for $3$-dimensional geometries. In particular, there are convex geometries $\mbf P_n$ with the $dim(\mbf P_n)=3$ and $cdim(\mbf P_n)=n+1$.

\section{Existence of Convex Geometry with cdim=3 Not Representable by Spheres}\label{sec:cg_existence}

Our goal in this section is to prove the following result.

\begin{theorem}\label{main}
    There exists a convex geometry with $cdim = 3$ that is not representable by spheres in any $\mathbb{R}^t$.
\end{theorem}

To achieve this, we provide here an explicit construction of a convex geometry with $cdim=3$, such that its poset of join-irreducibles is isomorphic to $\mbf n^3$. Representing this convex geometry by spheres in $\R^t$ would also result in an inclusion representation of $\mbf n^3$ by those spheres, as shown in Proposition \ref{bridge}, which we know for large enough $n$ is not possible by Theorem \ref{thm:FFT_poset}. Hence, this would suffice in proving the existence of a convex geometry not representable by spheres in the space of any dimension.

\

For the remainder of this section, we set the following notation. Fix $X = \{0, 1, 2, \ldots, n-1\}^3 \subseteq \mathbb{Z}^3$, and set $\mbf P$ to be the poset $\mbf P = (X, \leq)$, where $\leq$ is the natural ordering induced from $\mbf n$ in the manner of Definition \ref{def:n_poset}.

We now consider linear extensions of this natural ordering in $\mbf P.$ In particular, we examine lexicographic orderings on the set $X$.

\begin{definition}[$123$-lex-ordering]
    We define the poset $(X, \preceq_1)$. Given any two points $x = (x_1 , x_2  , x_3)$ and $y= (y_1  , y_2 , y_3)$ in $X$, we say
    \begin{align*}
        (x_1 , x_2  , x_3 ) \preceq_1 (y_1  , y_2 , y_3) &\text{ iff } x_1 \leq y_1 \\
        &\text{OR } x_1 = y_1, x_2 \leq y_2 \\
        &\text{OR } x_1 = y_1 , x_2 = y_2 , x_3 \leq y_3 
    \end{align*}
    This is precisely lexicographic ordering where we prioritize the values of the first coordinates, and if those are equal then the second coordinate, and finally the third.
\end{definition}

We can change the order of axes comparisons to get analogous but different lexicographic orderings as follows.

\begin{definition}[$231$-lex-ordering]
    We define the poset $(X, \preceq_2)$. Given any two points $x = (x_1 , x_2  , x_3)$ and $y= (y_1  , y_2 , y_3)$ in $X$, we say
    \begin{align*}
        (x_1 , x_2  , x_3 ) \preceq_2 (y_1  , y_2 , y_3) &\text{ iff } x_2 \leq y_2 \\
        &\text{OR } x_2 = y_2, x_3 \leq y_3 \\
        &\text{OR } x_2 = y_2 , x_3 = y_3 , x_1 \leq y_1
    \end{align*}
    This time we look first at the second coordinates, then the third, and finally the first.
\end{definition}

\begin{definition}[$312$-lex-ordering]
    We define the poset $(X, \preceq_3)$. Given any two points $x = (x_1 , x_2  , x_3)$ and $y= (y_1  , y_2 , y_3)$  in $X$, we say
    \begin{align*}
        (x_1 , x_2  , x_3 ) \preceq_3 (y_1  , y_2 , y_3) &\text{ iff } x_3 \leq y_3 \\
        &\text{OR } x_3 = y_3, x_1 \leq y_1 \\
        &\text{OR } x_3 = y_3 , x_1 = y_1 , x_2 \leq y_2
    \end{align*}
    This time we look first at the third coordinates, then the first, and finally the second.
\end{definition}

It is not hard to see that the lexicographic orderings are not just partial orderings, but rather organize the elements of $X$ into a chain. For example, if we consider $\mbf 2^3$, using $123$-lex-ordering instead gives us the following chain:

\begin{center}
    \begin{tikzpicture}
        \draw[-](0,0)--(10.5,0);
    
        \filldraw[black] (0,0) circle (2pt);
        \node at (0,.4) {(0,0,0)};
        
        \filldraw[black] (1.5,0) circle (2pt);
        \node at (1.5,.4) {(0,0,1)};
        
        \filldraw[black] (3,0) circle (2pt);
        \node at (3,.4) {(0,1,0)};
        
        \filldraw[black] (4.5,0) circle (2pt);
        \node at (4.5,.4) {(0,1,1)};
    
        \filldraw[black] (6,0) circle (2pt);
        \node at (6,.4) {(1,0,0)};
    
        \filldraw[black] (7.5,0) circle (2pt);
        \node at (7.5,.4) {(1,0,1)};
    
        \filldraw[black] (9,0) circle (2pt);
        \node at (9,.4) {(1,1,0)};
    
        \filldraw[black] (10.5,0) circle (2pt);
        \node at (10.5,.4) {(1,1,1)};
    \end{tikzpicture}
\end{center}

Now, we can use precisely the orderings $\preceq_1, \preceq_2$ and $\preceq_3$ (from Definitions $2, 3$ and $4$) -- respectively being the 123-lex-ordering, 231-lex-ordering, and 312-lex-ordering -- to generate a convex geometry following Remark \ref{rem:linearCG} and Theorem \ref{CGjoin}. More specifically, each of these three linear orders can be thought of as a monotone convex geometry on $X$ in the manner of Remark \ref{rem:linearCG}: so we obtain three linear geometries which we denote $(X,\mc F_1), (X,\mc F_2), (X,\mc F_3)$. Taking the join of these three linear geometries in the manner of Definition \ref{join} results in another convex geometry by Theorem \ref{CGjoin}. Finally, by Definition \ref{cdim}, this resultant convex geometry has $cdim \leq 3$. Call this geometry $(X, \mc F)$, with $X$ the base set and $\mc F$ as the closure space. Let us now prove that the poset of join-irreducibles of $\mc F$, denoted by $(\mathrm{Ji}(\mc F), \subseteq)$, is isomorphic to $\mbf P=(X,\leq)$.

It is well known that any join-irreducible in a standard closure system is precisely the minimal closed set generated by some singleton $x$ from base set $X$. Indeed, in standard closure system there is a one-to-one correspondence between elements of the base set and join-irreducibles. Denote by $J(x)$ the join-irreducible convex set in $\mc F$ corresponding to $x$. Recall the procedure of generating a convex geometry of $cdim\leq m$ as the intersection of closed sets of $m$ linear convex geometries on $X$: this is how we constructed $(X, \mc F)$ from our three linear lexicographic geometries. It follows that the smallest closed set generated by $x\in X$, equivalently $J(x)$, is the intersection of minimal closed sets generated by $x$ in each linear geometry. In summary, our observations in this paragraph state the following:

\begin{proposition}
  $J(x)$ will be obtained by intersecting the smallest convex set containing $x$ in each of the three generating linear geometries.  
\end{proposition}

Note importantly that the smallest convex set containing $x$ in the linear geometry $(X, \mc F_i)$ is precisely the down-set with respect to $\preceq_i$. That is, the minimal convex set containing $x$ in the $i$th of the three generating geometries is exactly $\{y \in X : y \preceq_i x\}$. The intersection over $i$ gives us $J(x)$.

So, we know there is a one-to-one correspondence between elements of $X$ and $\mathrm{Ji}(\mc F)$, and we also understand the nature of this correspondence. The final two propositions show that the correspondence preserves ordering, and hence conclude proving that $(\mathrm{Ji}(\mc F),\subseteq)$ is isomorphic to $\mbf P = (X,\leq)$.

\begin{proposition}\label{prop:isom1}
    If $x \leq y$ in $\mbf P$, then $J(x) \subseteq J(y)$.
\end{proposition}
\begin{proof}
    It suffices to show that $x \preceq_i y$ for each $i$, because of how we obtain $J(x)$ and $J(y)$ this would clearly imply $J(x) \subseteq J(y)$. 

    Write $x = (x_1, x_2, x_3)$ and $y = (y_1, y_2, y_3)$. We are told $x \leq y$, so we must have $x_i \leq y_i$, but this directly implies $x \preceq_i y$ by definition of the lex-ordering. This holds for any $i=1,2,3$, and so finishes our proof.
\end{proof}

\begin{proposition}\label{prop:isom2}
    If $x,y$ are incomparable in $\mbf P$, then $J(x), J(y)$ are incomparable (in terms of set-inclusion).
\end{proposition}
\begin{proof}
Write $x = (x_1, x_2, x_3)$ and $y = (y_1, y_2, y_3)$. For $x,y$ to be incomparable in $P$, we must precisely have the state that $x_i < y_i$, and $y_j < x_j$, for some $i \neq j$. Without loss of generality, assume that $x_1 < y_1$ and $y_2 < x_2$.

Now, because $x_1 < y_1$, we have $x \prec_1 y$ by definition of $123$-lex-ordering. This means $y \notin J(x)$. Similarly, because $y_2 < x_2$, we have $y \prec_2 x$. This further implies that $x \notin J(y)$.

So, we have $y \notin J(x)$ and $x \notin J(y)$, but of course $x \in J(x)$ and $y \in J(y)$, and we have hence proved that $J(x), J(y)$ are incomparable, when ordered by set-inclusion.
\end{proof}

Finally, we are ready to provide a quick proof of Theorem~\ref{main}.

\begin{proof}[Proof of Theorem~\ref{main}]
    For a given $n$, we constructed in this section a convex geometry $(X, \mc F)$. We showed that the join-irreducibles of $\mc F$, denoted $\mathrm{Ji}(F)$, are in bijection with elements of $P$, and Propositions \ref{prop:isom1} and \ref{prop:isom2} further show that they are in fact isomorphic as posets.

    Now, by Proposition~\ref{bridge} we know that any representation by spheres of $(X, \mc F)$ will result in $\mathrm{Ji}(\mc F)$, and hence, $\mbf P$, being a sphere order. But Theorem~\ref{thm:FFT_poset} tells us that for some $n_0$, if $n>n_0$ then $\mbf P$ is not a sphere order. So, for large enough $n$, $(X, \mc F)$ cannot be representable by spheres.

    Finally, $(X, \mc F)$ by construction has $cdim \leq 3$. However, in \cite{Cz14} it is shown that all convex geometries of $cdim \leq 2$ are representable by circles on the plane. So, in fact for $n>n_0$, $(X, \mc F)$ must have $cdim = 3$, which proves Theorem~\ref{main}.
\end{proof}

\section{Ellipsoid Orders}

In the survey on geometric containment orders, Fishburn and Trotter \cite{FT99} discuss various results related to poset representations via containment orders of different geometric objects such as angles, polygons and spheres. They also mention that representations by ellipsoids were not intensely studied. In \cite{FT02} they have shown that any 2-dimensional poset can be represented as a containment order by a family of \emph{similar} ellipsoids that share the same center. 

As for to convex geometries, the following result was shown by J. Kincses.

\begin{theorem}\cite{Kin17}\label{Kinc}
Any finite convex geometry with convex dimension $t$ can be represented in $\mathbb{R}^t$
with ellipsoids which are arbitrary close to a sphere. 
\end{theorem}

The construction of ellipsoids in Theorem \ref{Kinc} does not assume similarity of ellipsoids, but all of them contain a unit sphere in $\mathbb{R}^t$ and are themselves contained in the sphere of radius $s$. Taking $s$ close to $1$ allows to make them close to the unit sphere.

Another result in the same paper \cite{Kin17} formulates an Erd\"os-Szekeres type of obstruction from Dobbins et al.\! \cite{DHH16} that shows that not all convex geometries are represented by ellipses on the plane.

We also mention the result in \cite{Poly20-2} about representation of all geometries on 5-element set by ellipses on the plane. Thus, it would be interesting to learn what is the smallest size of base set of a geometry so that representation by ellipsoids exists only in $\R ^k$ with $k>2$.

We can now connect representation of geometries by ellipsoids and the notion of ellipsoid order. We mention that Proposition \ref{bridge} could be formulated for ellipsoids in place of spheres, which connects representation of convex geometries and posets.

\begin{theorem}
    Every finite poset $\mbf P=(X,\leq)$ is an ellipsoid order.
\end{theorem}
\begin{proof}
    First, we show that every poset is realized as the poset of join-irreducible elements of some convex geometry. Indeed, starting from $\mbf P=(X,\leq)$ we can build the lattice $\mbf D= \down(X,\leq)$
    of down-sets of $\mbf P$, which is a distributive lattice by Birkhoff's Theorem~\ref{thm:birkhoff}. Every finite distributive lattice is a convex geometry, since related closure operator satisfies the the Anti-Exchange Property. (See more general description of lattice properties of finite convex geometries in \cite[Theorem 5-2.1]{AdNa16}). Since the join-irreducibles of this convex geometry are precisely the down-sets of singletons, i.e.\! down-sets of the form $\downarrow x$ for $x \in X$, we can conclude in the manner of Example~\ref{exam:down_set} that $\mbf P$ is realized by join-irreducible elements of this convex geometry.

    By Theorem \ref{Kinc} this geometry can be represented by ellipsoids in some space $\mathbb{R}^t$, where $t=cdim(\mbf D)$. In particular, the set of join-irreducible elements of $\mbf D$, which is isomorphic to $\mbf P$, will provide the ellipsoid containment representation of $\mbf P$.
\end{proof}

\printbibliography

\

\begin{flushright}\it{
KIRA ADARICHEVA\\
Hofstra University\\
Hempstead, New York, USA\\
\textit{E-mail address:} \href{mailto:kira.adaricheva@hofstra.edu}{\texttt{kira.adaricheva@hofstra.edu}}

\

ARAV AGARWAL\\
Bowdoin College\\
Brunswick, Maine, USA\\
\textit{E-mail address:} \href{mailto:aagarwal@bowdoin.edu}{\texttt{aagarwal@bowdoin.edu}}

\

NA'AMA NEVO\\
Northeastern University\\
Boston, Massachusetts, USA\\
\textit{E-mail address:} \href{mailto:nevo.n@northeastern.edu}{\texttt{nevo.n@northeastern.edu}}
 
}\end{flushright}

\end{document}